\documentclass[12pt, reqno]{amsart}
\usepackage{amssymb, amstext, amscd, amsmath}
\usepackage{color}

\usepackage{xy}
\xyoption{all}

\makeatletter
\def\@cite#1#2{{\m@th\upshape\bfseries%
[{#1\if@tempswa{\m@th\upshape\mdseries, #2}\fi}]}} \makeatother
\theoremstyle{plain}
\newtheorem{thm}[subsection]{Theorem}
\newtheorem{cor}[subsection]{Corollary}
\newtheorem{prop}[subsection]{Proposition}
\newtheorem{lem}[subsection]{Lemma}
\theoremstyle{definition}
\newtheorem{rem}[subsection]{Remark}

\newcommand{\bC}{{\mathbb{C}}}

\newcommand{\bM}{{\mathbb{M}}}
\newcommand{\bN}{{\mathbb{N}}}

\newcommand{\bQ}{{\mathbb{Q}}}
\newcommand{\bR}{{\mathbb{R}}}

\newcommand{\bT}{{\mathbb{T}}}

\newcommand{\bZ}{{\mathbb{Z}}}

\newcommand{\bm}{{\mathbf{m}}}
\newcommand{\bn}{{\mathbf{n}}}
\newcommand{\bt}{{\mathbf{t}}}

\renewcommand{\O}{{\mathcal{O}}}

\renewcommand{\S}{{\mathcal{S}}}

\newcommand{\fA}{{\mathfrak{A}}}

\newcommand{\fF}{{\mathfrak{F}}}

\newcommand{\fM}{{\mathfrak{M}}}

\renewcommand{\phi}{\varphi}
\newcommand{\upchi}{{\raise.35ex\hbox{\ensuremath{\chi}}}}

\newcommand{\qforal}{\quad\text{for all}\quad}

\newcommand{\Ad}{\operatorname{Ad}}

\newcommand{\id}{{\operatorname{id}}}

\newcommand{\spn}{\operatorname{span}}

\newcommand{\ca}{\mathrm{C}^*}

\newcommand{\Fn}{\mathbb{F}_n^+}
\newcommand{\Fm}{\mathbb{F}_m^+}

\newcommand{\Fth}{\mathbb{F}_\theta^+}

\newcommand{\mt}{\varnothing}

\newcommand{\ol}{\overline}

\begin{document}
\title[ ]
{Type III von Neumann Algebras \\ associated with $\O_\theta$ }
\author[D. Yang]
{Dilian Yang}
\address{Dilian Yang,
Department of Mathematics $\&$ Statistics, University of Windsor, Windsor, ON
N9B 3P4, CANADA} \email{dyang@uwindsor.ca}

\begin{abstract}
Let $\Fth$ be a 2 graph generated by $m$ blue edges and $n$ red edges,
and $\omega$ be the distinguished faithful state
associated with its graph C*-algebra $\O_\theta$.

In this paper, we characterize the factorness of the von Neumann algebra
$\pi_\omega(\O_\theta)''$ induced from the GNS representation of $\omega$ under a certain condition.
Moreover, when $\pi_\omega(\O_\theta)''$ is a factor, then it is 
of type III$_{m^{-\frac{1}{b}}}$ (or III$_{n^{-\frac{1}{a}}}$) if $\frac{\ln m}{\ln n}\in\bQ$,
where $a,b\in\bN$ with $\gcd(a,b)=1$ satisfy $m^a=n^b$,
and of type III$_1$ if $\frac{\ln m}{\ln n}\not\in\bQ$.
In the case of $\theta$ being the identity permutation, our condition turns out to be redundant.

On the way to our main results, we also obtain the structure of the fixed point algebra $\O_\theta^\sigma$ of the modular action
$\sigma$ from $\omega$.
This could be useful in proving the redundancy of our extra condition.

\end{abstract}

\thanks{The author was supported in part by an NSERC Discovery grant.}

\date{}
\maketitle

\section{Introduction}
\label{Intr}

In this paper, continuing \cite{DPY1, DPYdiln, DYperiod, Yang}, we study rank 2 graph algebras with a single vertex.
Roughly speaking, a rank 2 graph with a single vertex $\Fth$ is a 2-dimensional directed graph generated by blue and red edges
satisfying some commutation relations.
More precisely, consider two types of letters $e_1,\ldots, e_m$ and $f_1,\ldots, f_n$, which can be regarded
as the $m$ blue and $n$ red edges, respectively. Let $\theta$ be a given
permutation of $m\times n$. Construct a unital semigroup $\Fth$ with generators $e_i$'s and $f_j$'s, such that
it is free in $e_i$'s and free in $f_j$'s, and has the commutation relations
$e_if_j=f_{j'}e_{i'}$ where $\theta(i,j)=(i',j')$ for all $1\le i\le m$, $1\le j\le n$.
Those graphs are a special and important class of rank $k$ graphs (also called higher rank graphs), which
have been recently attracting a great deal of attention. See, e.g., \cite{KumPask, PaskRRS, Raeburn, RaeSimYee1,
RaeSimYee2} and the references therein.

Rank 2 graphs were first investigated by Power in \cite{P1}, where the associated non-commutative Toeplitz algebras
were studied in detail.
Then with Davidson and Power, we systematically studied the representation theory and
dilation theory of rank 2 graph algebras in \cite{DPY1, DPYdiln}.
We gave a general group construction for their atomic representations and completely classified them in \cite{DPY1}.
In order to study representations, it is often useful to consider dilation theory. In \cite{DPYdiln}, it was shown that
every defect free row contractive representation has a unique minimal *-dilation. Those two papers are basically concerned
with the non-selfadjoint operator algebras associated with $\Fth$. In \cite{DYperiod}, Davidson and I
gave some characterizations of the aperiodicity of $\Fth$, and described the structure of its graph
C*-algebra $\O_\theta$.

Since rank 2 graph C*-algebras $\O_\theta$ are a generalization of Cuntz algebras,
motivated by the work \cite{Cun},
we considered the (unital) endomorphisms and modular theory of $\O_\theta$ in \cite{Yang}. It turns out that
many results in Cuntz algebras are naturally generalized to $\O_\theta$. We proved that
there is a bijection between endomorphisms of $\O_\theta$ and its unitary pairs with
a twisted property. The twisted property comes up naturally as it decodes the commutation
relations of $\Fth$. As in Cuntz algebras, the bijection plays a vital role in the analysis of
properties of endomorphisms.

Among endomorphisms of $\O_\theta$, gauge automorphisms play a special role.  The integration of
gauge automorphisms over the 2-torus $\bT^2$ gives a faithful expectation $\Phi$ of $\O_\theta$
onto the fixed point algebra $\fF$ of the gauge action $\gamma$. It is known that $\fF$ is of $(mn)^\infty$ type UHF algebra.
Particularly, $\fF$ has a unique tracial state, denoted by $\tau$. Then $\tau$
and $\Phi$ induce the distinguished
faithful state $\omega$ of $\O_\theta$, namely, $\omega:=\tau\circ \Phi$. Let $\pi_\omega(\O_\theta)''$
denote the von Neumann algebra induced from the GNS representation of $\omega$.
In \cite{Yang}, we obtained an explicit formula for the modular automorphisms $\{\sigma_t:t\in \bR\}$ of $\pi_\omega(\O_\theta)''$
associated with $\omega$.
Identifying $\O_\theta$ with $\pi_\omega(\O_\theta)$, we obtain a C*-dynamical system $(\O_\theta, \bR, \sigma)$.
It was proved there that if $\frac{\ln m}{\ln n}\not\in\bQ$, then $\omega$ is the unique $\sigma$-KMS state over $\O_\theta$,
and $\pi_\omega(\O_\theta)''$ is an AFD factor of type III$_1$.

As remarked in \cite{Yang}, $\frac{\ln m}{\ln n}\not\in\bQ$ implies that $\Fth$ is aperiodic;
however, the converse is not true (cf. \cite{DYperiod}).
Also, it was shown in \cite{DYperiod} that $\Fth$ is aperiodic iff $\O_\theta$ is simple.
So one naturally wonders if the following holds true: $\pi_\omega(\O_\theta)''$ is a factor iff $\Fth$ is aperiodic.
From the above, it suffices to see if the aperiodicity of $\Fth$ with $\frac{\ln m}{\ln n}\in\bQ$ implies the factorness of $\pi_\omega(\O_\theta)''$.
If this is the case, one could also further ask its type.

In this paper, we partially answer the above questions.
We show that, under the assumption that the fixed point algebra $\O_\theta^\sigma$ of the modular automorphism
action $\sigma$
has a unique tracial state,
$\pi_\omega(\O_\theta)''$ is a factor $\Longleftrightarrow$ $\Fth$ is aperiodic $\Longleftrightarrow$
$\omega$ is the unique $\sigma$-KMS state over $\O_\theta$ (Theorem \ref{T:factor}).
Moreover, there is a dichotomy for the type of $\pi_\omega(\O_\theta)''$:
if $\frac{\ln m}{\ln n}\in\bQ$, it is of type III$_{m^{-\frac{1}{b}}}$ (or III$_{n^{-\frac{1}{a}}}$),
where $a,b\in\bN$ with $\gcd(a,b)=1$ satisfy $m^a=n^b$; and it is of type III$_1$ if $\frac{\ln m}{\ln n}\not\in\bQ$
(Theorem \ref{T:apefactor}).
Of course, if $\frac{\ln m}{\ln n}\not\in\bQ$, then our assumption is redundant as $\O_\theta^\sigma=\fF$.
For the case where $\frac{\ln m}{\ln n}\in\bQ$, the structure theorem of $\O_\theta^\sigma$
(Theorem \ref{T:crossprod}), which says that $\O_\theta^\sigma$ is a crossed product of $\fF$
by $\bZ$, might be useful in proving that our assumption could be redundant
(cf., e.g.,  \cite{Bedos1, Bedos2, Thom}, and also refer to Remark \ref{R:assump}).

The rest of this paper is organized as follows. In Section \ref{S:Pre}, we will give some preliminaries about rank 2 graph algebras,
which will be used later. The fixed point algebra 
$\O_\theta^\sigma$ will be studied in Section \ref{S:cropro}. It is shown that
if $\Fth$ is aperiodic, then $\O_\theta^\sigma$ is either $\fF$, or a crossed product
of $\fF$ by $\bZ$, and so particularly it is simple.
In Section \ref{S:KMS}, we analyze
the KMS-states of the C*-dynamical system $(\O_\theta, \bR, \sigma)$, and prove that $\omega$ is the unique $\sigma$-KMS state
of $\O_\theta^\sigma$
if $\O_\theta^\sigma$ has a unique tracial state. However,
if $\theta$ is the identity permutation (i.e., $\theta=\id$),
then our assumption is redundant. Making use of the results in Section \ref{S:KMS},
we are able to determine the type of $\pi_\omega(\O_\theta)''$ in Section \ref{S:type}. 
It is of either type III$_\lambda$ ($0<\lambda<1$) or type III$_1$, which is dependent of whether $\frac{\ln m}{\ln n}\in\bQ$ or not.
The type of $\pi_\omega(\O_\id)''$
is obtained without any further assumptions. But the proof needs a different approach.

\section{Preliminaries}\label{S:Pre}

A rank 2 graph on a single vertex is a unital semigroup
$\Fth$, which is generated by $e_1 ,\ldots, e_m$ and
$f_1 , \ldots, f_n$.
The identity is denoted as $\mt$. There are no relations among the $e_i$'s, so they
generate a copy of the free semigroup $\Fm$ on $m$ letters; and there are no
relations on the $f_j$'s, so they generate a copy of $\Fn$. There are
\textit{commutation relations} between the $e_i$'s and $f_j$'s
given by a permutation $\theta$ in $S_{m\times n}$ of $m \times  n$:
\[
e_if_j=f_{j'}e_{i'} \quad \mbox{where}\quad \theta(i, j ) = (i' , j').
\]
It is known that any word $w\in\Fth$ has fixed numbers of $e$'s and $f$'s
regardless of the factorization. So one can define the \textit{degree} of  $w$ as $d(w)=(k,l)$ if there are $k$ $e$'s and
$l$ $f$'s, and the \textit{length} of $w$ as $|w|=k+l$. Moreover, because of the commutation relations, one can
write $w$ according to any prescribed pattern of $e$'s and $f$'s as long as the degree is $(k,l)$.

Let us recall from \cite{KumPask} that the \textit{graph C*-algebra $\O_{\theta}$} of $\Fth$ is
the universal $C^*$-algebra generated by
a family of isometries $\{s_u:u\in\Fth\}$ satisfying
$s_{\mt}=I$, $s_{uv}=s_us_v$ for all $u,\ v\in \Fth$,
and the defect free property:
\[
\sum_{i=1}^ms_{e_i}s_{e_i}^*=I=\sum_{j=1}^ns_{f_j}s_{f_j}^*.
\]

It is well-known that $s_us_v^*$ ($u, v\in \Fth$) are the standard generators of $\O_\theta$,
and that
$$\O_\theta=\overline{\spn}\{s_us_v^*:u,v\in\Fth\}$$
(\cite[Lemma 3.1]{KumPask}).
For convenience, as in \cite{Yang}, we extend the degree map $d$ of $\Fth$ to
the standard generators of $\O_\theta$ by
$$d(s_us_v^*)=d(u)-d(v)\qforal u,v\in\Fth.$$

The universal property of $\O_{\theta}$ yields a family of
\textit{gauge automorphisms}
$\gamma_{\bt}$ for $\bt=(t_1, t_2) \in \bT^2$
given by
\[
 \gamma_{\bt}(s_w) =\bt^{d(w)} s_w\qforal w\in\Fth,
\]
where $\bt^{d(w)}=t_1^{n_1}t_2^{n_2}$ if $d(w)=(n_1,n_2)$.
Integrating around $\bT^2$ yields a faithful expectation
\[ \Phi(X) = \int_{\bT^2}  \gamma_{\bt}(X) \,d\bt\]
onto the fixed point algebra $\O_{\theta}^\gamma$ of $\gamma$.

It turns out that
$$\fF:=\O_{\theta}^\gamma=\ol{\bigcup_{k\ge 1} \fF_k}$$
is an $(mn)^\infty$-UHF algebra,
where
$$\fF_k=\ol\spn\{ s_us_v^* : d(u)=d(v)=(k,k),\ k\in\bN\}$$
is the full matrix algebra $\bM_{(mn)^k}$.
Hence there is a unique faithful tracial state, say $\tau$, on $\fF$.
With $\Phi$, $\tau$ induces a distinguished faithful state $\omega$ on $\O_\theta$ via
$$
\omega(X)=\tau(\Phi(X))\quad \text{for all}\quad X\in\O_\theta.
$$
Let $\pi_{\omega}(\O_\theta)''$ be the von Neumann algebra generated by the (faithful) GNS representation of
$\omega$. When no confusion is caused, we shall omit the index $\omega$ and write $\pi(\O_\theta)''$, instead of $\pi_\omega(\O_\theta)''$, for short.
Then $\omega$ can be (uniquely) extended to a normal faithful state on $\pi(\O_\theta)''$. The extension will be
still denoted by $\omega$.

\smallskip

If $\Fth$ is periodic, then it is well-studied in \cite{DYperiod}. So, throughout the remainder of this paper, we assume that
$\Fth$ is aperiodic, unless otherwise specified. 
Equivalently, its graph C*-algebra $\O_\theta$ is simple. 
In particular, this implies that $m>$ and $n>1$.


\section{The structure of the fixed point algebra $\O_\theta^\sigma$}
\label{S:cropro}

From \cite{Yang},
the parameter group of modular automorphisms of $\omega$
on $\pi(\O_\theta)''$ was given explicitly by the following:
\begin{align}
\label{E:sigma}
\sigma_t(\pi_\omega(s_us_v^*))
=\bn^{it(d(v)-d(u))}\pi_\omega(s_us_v^*)\quad (t\in \bR).
\end{align}
Here the multi-index notation is used: for $(x,y)\in\bC^2$, $\bn^{(x,y)}=m^xn^y$. Also, recall that
$m^x=\exp(x\ln m)$ and similarly for $n^y$.

In particular, we have
$$
\sigma_t(\pi_\omega(s_{e_i}))=m^{-it}\pi_\omega(s_{e_i}),\quad
\sigma_t(\pi_\omega(s_{f_j}))=n^{-it}\pi_\omega(s_{f_j})
$$
for all $i=1,...,m$, $j=1,...,n$.
Since $\omega$ is faithful, we can identify $\O_\theta$ with $\pi_\omega(\O_\theta)$.
Then $\{\sigma_t:t\in \bR\}$ yields a parameter group of automorphisms of $\O_\theta$.
So we have a C*-dynamical system $(\O_\theta, \bR, \sigma)$, which is simply denoted by $(\O_\theta, \sigma)$.

In \cite{Yang}, we see that the fixed point algebra $\O_\theta^\sigma$ of $\sigma$ is nothing but $\fF$ if $\frac{\ln m}{\ln n}\not\in\bQ$.
In order to study $\O_\theta^\sigma$ in general, we only need to consider the case of $\frac{\ln m}{\ln n}\in\bQ$. That is,
\begin{align}
\mbox{there are}\ a, b\in\bN \ \mbox{with}\ \gcd(a,b)=1 \ \mbox{such that}\ m^a=n^b.
\tag{$*$}
\end{align}
So in the remainder of this section, we assume that ($*$) holds true.
By $\bm^a$ (resp. $\bn^b$), we mean the set of all words in $e_i$'s
(resp. $f_j$'s) of length $a$ (resp. $b$).
Let $\jmath: \bm^a\to \bn^b$ be a fixed bijection (such a bijection exists because of {($*$)} ). Define
$$U=\sum_{u\in\bm^a}s_{e_u}s_{f_{\jmath(u)}}^*.$$
Then it is easy to see that $U$ is a unitary operator. In fact, due to the defect free property, one has
$$UU^*
=\sum_{u\in\bm^a}s_{e_u}s_{f_{\jmath(u)}}^* \sum_{v\in\bm^a}s_{f_{\jmath(v)}}s_{e_v}^*
=\sum_{u\in\bm^a}s_{e_u}s_{e_{u}}^*
=I,
$$
and similarly $U^*U=I$.
Define an automorphism $\rho$ of $\fF$ by $\rho=\Ad(U)$:
$$\rho(X)=UXU^*\quad \text{for all} \quad X\in \fF.$$

We are in a position to prove the main theorem of this section,
which gives the structure of the fixed point algebra $\O_\theta^\sigma$ of $\sigma$.

\begin{thm}
\label{T:crossprod}
Let $\rho, a, b$ be the same as above.
Then
$\O_\theta^\sigma \cong \fF\times_\rho \bZ$. Moreover, $\O_\theta^\sigma$ is simple.
\end{thm}

\begin{proof}
We first claim that
\begin{align}
\label{E:fixe1}
\O_\theta^\sigma=\ca\{X\in\O_\theta: d(X)\in\bZ (a,-b)\}.
\end{align}
By the definition of $\sigma$, it is easy to check that,  the right hand side of \eqref{E:fixe1} is included in its
left hand side. In order to show the reverse inclusion, after a moment's thought,
it suffices to verify that the generators of $\O_\theta^\sigma$ are contained in the right hand side of \eqref{E:fixe1}. 
To this end, let $X=s_us_v^*\in\O_\theta^\sigma$. Since $\sigma_t(X)=X$ for all $t\in\bR$, we have
$\bn^{it(d(v)-d(u))}=1$ for all $t\in\bR$. Assume that $d(v)-d(u)=(x,y)\in\bZ^2$. Then
a simple calculation shows that
$bx+ay=0$. Since $\gcd(a,b)=1$, we have
$(x,y)\in\bZ (a,-b)$.
This proves our claim.

In what follows, we show that $\O_\theta^\sigma$ is the C*-algebra generate by $\fF$ and $U$:
\begin{align}
\label{E:fixcros}
\O_\theta^\sigma=\ca(\fF, U).
\end{align}
From \eqref{E:fixe1}, evidently, $\ca(\fF, U)\subseteq \O_\theta^\sigma$. In order to show
$\O_\theta^\sigma\subseteq\ca(\fF, U)$, thanks to \eqref{E:fixe1}, we only need to prove
 that all elements $X\in\O_\theta$ with $d(X)\in\bZ(a,-b)$ belong to $\ca(\fF,U)$.
For this, we first assume
$X=s_{e_u}s_{f_v}^*$ with $(|u|,|v|)=k(a,b)$ for some $k\in \bN$.
Then we prove $X\in\fF U^k$ by induction.
Indeed, if $k=1$, then one has
\begin{align*}
X=s_{e_u}s_{f_v}^*
   =s_{e_u}s_{e_{\jmath^{-1}(v)}}^*s_{e_{\jmath^{-1}(v)}}s_{f_v}^*
   =s_{e_u}s_{e_{\jmath^{-1}(v)}}^*U
   \in\fF U,
 \end{align*}
where the last equality used the identity $s_{e_{\jmath^{-1}(v)}}^* s_{e_u}=\delta_{u, \jmath^{-1}(v)}$. Here 
$\delta$, as usual, is the Dirac delta function.
Now assume that
$s_{e_u}s_{f_v}^*$ with $(|u|,|v|)=k(a,b)$ belongs to $\fF U^k$.
Consider $X=s_{e_\alpha}s_{f_\beta}^*$ with $(|\alpha|,|\beta|)=(k+1)(a,b)$.
Then
$X=s_{e_{u_0}}s_{e_u}s_{f_v}^*s_{f_{v_0}}^*$
with $(|u|,|v|)=k(a,b)$ and $(|u_0|,|v_0|)=(a,b)$,
where $\alpha=u_0u$ and $\beta=v_0v$.
By the inductive assumption, one has
\begin{align*}
X=s_{e_\alpha}s_{f_\beta}^*
&\in s_{e_{u_0}}\fF U^k s_{f_{v_0}}^*\\
&=s_{e_{u_0}}\fF U^ks_{e_{\jmath^{-1}(v_0)}}^*s_{e_{\jmath^{-1}(v_0)}} s_{f_{v_0}}^*\\
&=s_{e_{u_0}}\fF U^ks_{e_{\jmath^{-1}(v_0)}}^*U
      \quad (\text{using} \ s_{e_{\jmath^{-1}(v)}}^* s_{e_u}=\delta_{u, \jmath^{-1}(v)})\\
&=s_{e_{u_0}}\fF U^ks_{e_{\jmath^{-1}(v_0)}}^*{U^*}^k\cdot U^{k+1}\\
&\subseteq \fF U^{k+1}.
\end{align*}
Therefore we obtain that, for all $n\in\bN$, $s_{e_u}s_{f_v}^*$ with $(|u|,|v|)=n(a,b)$ belong to $\fF U^n$.

From \eqref{E:fixe1}, one can assume that the most general form of the standard generator $X$
of $\O_\theta^\sigma$ is given by
$X=s_{w_1}s_{e_u}s_{f_v}^*s_{w_2}^*$ where $w_1,w_2\in\Fth$ with $d(w_1)=d(w_2)$ and $(|u|,|v|)=k(a,b)$.
From the above analysis, $s_{e_u}s_{f_v}^*\in \fF U^k$. Then
\begin{align*}
X\in s_{w_1}\fF U^k s_{w_2}^*=s_{w_1}\fF U^k s_{w_2}^*{U^*}^k \cdot U^k\subseteq \fF U^k.
\end{align*}
Therefore
 $\O_\theta^\sigma \subseteq \ca(\fF,U)$, implying
 $\O_\theta^\sigma = \ca(\fF,U)$.

Finally we prove
\begin{align}
\label{E:FU}
\ca(\fF, U)\cong \fF\times_\rho \bZ.
\end{align}
Once this is done, it immediately follows from \eqref{E:fixcros} and \eqref{E:FU} that $\fF\times_\rho \bZ\cong\O_\theta$,
and particularly that $\O_\theta^\sigma$ is simple.

In order to prove \eqref{E:FU}, notice that $\rho$ is an outer automorphism of $\fF$. Indeed, if $\rho$ were inner, then
there would be a unitary operator $V\in\fF$ such that $\rho(X)=VXV^*$. Thus
$UXU^*=VXV^*$ for all $X\in\fF$. So $V^*UX=XV^*U$. That is,
$V^*U$ is in the relative commutant, $\fF'\cap \O_\theta$, of $\fF$. But since $\Fth$ is aperiodic
(from our long standing assumption),
$V^*U$ is a scalar by \cite{Yang}. This implies $U\in\fF$, a contradiction.
The completely same reasoning gives that every $\rho^n$ ($n\in \bZ, n\ne 0$) is outer.
Equivalently, $\rho$ is an aperiodic automorphism of $\fF$.

Since $\fF$ is simple and $\rho$ is aperiodic, it is well-known that
$\fF\times_\rho \bZ$ is also simple
(cf., e.g., \cite{Bedos1, Krish}).
From the universal property of crossed product $\fF\times_\rho \bZ$, there is a *-homomorphism
from $\fF\times_\rho \bZ$  onto $\ca(\fF, U)$. The simplicity of $\fF\times_\rho \bZ$ implies that the *-homomorphism
is also injective. This concludes that $\fF\times_\rho \bZ$  and $\ca(\fF, U)$ are *-isomorphic.
We are done.
\end{proof}

\begin{rem}
From the definition of the above unitary operator $U$, it depends on the bijection $\jmath$ we have chosen. However, if $\tilde U$ is
another unitary operator given by another bijection $\tilde \jmath$, then we still have
$\O_\theta^\sigma = \ca(\fF,\tilde U)\cong \fF\times_{\tilde \rho} \bZ$,
where $\tilde \rho=\Ad \tilde U$.
This is because
$$\tilde U
=\sum_{u\in\bm^a}s_{e_u}s_{f_{\tilde \jmath(u)}}^*
=\sum_{u\in\bm^a}s_{e_u}s_{f_{\jmath(u)}}^* s_{f_{\jmath(u)}}s_{f_{\tilde \jmath(u)}}^*
=Us_{f_{\jmath(u)}}s_{f_{\tilde \jmath(u)}}^*
\in U\fF.
$$
Thus the choice of $U$ (i.e., $\jmath$) in Theorem \ref{T:crossprod} is not essential.

\end{rem}

It follows from Theorem \ref{T:crossprod} that

\begin{cor}
\label{C:indep}
$\O_\theta^\sigma$ is independent of $\theta$. More precisely,
for two permutations $\theta_1$ and $\theta_2$ of $m\times n$, we have
$\O_{\theta_1}^{\sigma^1}\cong \O_{\theta_2}^{\sigma^2}$, where
$\sigma^1$ and $\sigma^2$ are induced from the modular actions on $\pi(\O_{\theta_1})''$
and $\pi(\O_{\theta_2})''$, respectively.
\end{cor}

\section{KMS-States of $\O_\theta$}
\label{S:KMS}

In this section, we shall analyze the $\sigma$-KMS states for the C*-dynamical system $(\O_\theta, \sigma)$.
Using the properties of the identity permutation (i.e., $\theta=\id$), we first show that $\omega$
is the unique $\sigma$-KMS state of $\O_\theta$. However, in the general
case, we have to further assume that $\O_\theta^\sigma$ has a unique tracial state in order
to conclude the uniqueness of $\sigma$-KMS states. But from the structure of $\O_\theta^\sigma$
obtained in Section \ref{S:cropro},
it would be possible to prove that our assumption would be redundant (cf. \cite{Bedos1, Bedos2, Thom}).
Unfortunately, I am not able to prove this so far.

We first recall the definition of KMS-states for C*-dynamical systems from \cite[Chapter 5]{BraRob}.
Let $(\fA, \varrho)$ be a C*-dynamical system. The state $\varphi$ over $\fA$ is said to be a
\textit{$\varrho$-KMS state  at value $\beta\in\bR$}, or a \textit{$(\varrho, \beta)$-KMS state}, if
\begin{align}
\label{E:KMS}
\varphi(AB)=\varphi(B\varrho_{i\beta}(A))
\end{align}
for all $A,B$ entire for $\rho$.
A $\varrho$-KMS state at value $\beta=-1$ is simply called a \textit{$\varrho$-KMS state}.

For the C*-dynamical system $(\O_\theta, \sigma)$,
the following lemma says that $(\O_\theta, \sigma)$ has only $\sigma$-KMS states.

\begin{lem}
\label{L:beta-1}
If $\varphi$ is a $(\sigma, \beta)$-KMS state over $\O_\theta$, then $\beta=-1$.
So every KMS-state over $\O_\theta$ for $\sigma$ is a $\sigma$-KMS state.
\end{lem}

\begin{proof}
Notice that the restriction of $\sigma$ to the Cuntz algebra $\O_m=\ca(s_{e_1},...,s_{e_m})$,
say $\tilde\gamma$,
is nothing but a rescaling of the usual gauge action
of $\O_m$.
Since $\varphi$ is a $\sigma$-KMS state over $\O_\theta$ at value $\beta$,
its restriction to the Cuntz algebra $\O_m$
is a $(\tilde\gamma, \beta)$-KMS state over $\O_m$.
From the well-known uniqueness of KMS states of Cuntz algebras
of gauge actions, we have $\beta=-1$ (cf. \cite{OlePed}).
\end{proof}

In \cite{Yang}, it is shown that $\omega$ is the unique $\sigma$-KMS state over $\O_\theta$
when $\frac{\ln m}{\ln n}\notin\bQ$. The following result says this is also the case even if
$\frac{\ln m}{\ln n}\in\bQ$ when $\theta=\id$.

\begin{prop}
\label{P:uni_id}
$\omega$ is the unique $\sigma$-KMS state over $\O_\id$.
\end{prop}

\begin{proof}
Suppose that $\varphi$ is another $\sigma$-KMS state on $\O_\id$.
From the KMS condition \eqref{E:KMS},
we have $\varphi|_\fF=\omega|_\fF=\tau$, the unique tracial state on $\fF$.
To show $\varphi=\omega$,
it suffices to verify that $\varphi$ vanishes on all standard generators of $\O_\theta$,
which are not in $\fF$.

Since $\theta=\id$, $s_{e_i}$'s and $s_{e_i}^*$'s commute with $s_{f_j}$'s and $s_{f_j}^*$'s
(\cite{DYperiod}). This simple observation will be used frequently in the sequel without any further mention.
Now let $X:=s_{e_{u_1}f_{v_1}} s_{e_{u_2}f_{v_2}}^*\not\in\fF$.
From the KMS condition \eqref{E:KMS}, we obtain
\begin{align}
\nonumber
\varphi(X)
&=
\varphi(s_{e_{u_2}f_{v_2}}^*\sigma_{-i}(s_{e_{u_1}f_{v_1}})) \\
\nonumber
&=
m^{-|u_1|}n^{-|v_1|}\varphi(s_{f_{v_2}}^*s_{e_{u_2}}^*s_{e_{u_1}}s_{f_{v_1}})\\
\label{E:varphi}
&=
\begin{cases}
m^{-|u_1|}n^{-|v_1|}\varphi(s_{e_{u'}}^*s_{f_{v_2}}^*s_{f_{v_1}}), & \ \text{if}\  u_2=u_1u', \\
m^{-|u_1|}n^{-|v_1|}\varphi(s_{f_{v_2}}^*s_{f_{v_1}}s_{e_{u''}}), &\ \text{if}\  u_1=u_2 u'',\\
0, &\ \text{otherwise}.
\end{cases}
\end{align}
From \eqref{E:varphi}, in order to show $\varphi(X)=0$, it is sufficient to prove that
$\varphi(s_{f_v}^*s_{e_u})=0$, $\varphi(s_{f_v}s_{e_u})=0$, $\varphi(s_{f_v}s_{e_u}^*)=0$,
and $\varphi(s_{f_v}^*s_{e_u}^*)=0$
for all $u,v\in\Fth$, where at least one of $u,v$ with length nonzero. 
For convenience, set $s_{e_u}^{1}:=s_{e_u}$,   $s_{f_v}^{1}:=s_{f_v}$, $s_{e_u}^{-1}:=s_{e_u}^*$,
and $s_{f_v}^{-1}:=s_{f_v}^*$.

(i)
Neither of $|u|, |v|$ is 0. Then it follows from the KMS-condition that 
\begin{align*}
\varphi(s_{f_v}^{\mp 1}s_{e_u}^{\pm 1})
&=\varphi(s_{e_u}^{\pm 1}\sigma_{-i}(s_{f_v}^{\mp 1}))\ (\text{by}\ \eqref{E:KMS})\\
&=n^{\pm |v|}\varphi(s_{e_u}^{\pm 1}s_{f_v}^{\mp 1})\ (\text{by}\ \eqref{E:sigma}) \\
&=n^{\pm |v|}\varphi(s_{f_v}^{\mp 1}s_{e_u}^{\pm 1})\ (\text{as}\ \theta=\id) .
\end{align*}
This forces $\varphi(s_{f_v}^{\mp 1}s_{e_u}^{\pm 1})=0$, as $|v|\ne 0$.

\smallskip

(ii) One of $|u|, |v|$ is 0.
Without loss of generality, we assume that $|v|=0$.
Since $|u|\ne 0$ from our assumption, the following calculations give $\varphi(s_{e_u}^{\pm 1})=0$:
\begin{align*}
\varphi(s_{e_u}^{\pm 1})
&=\sum_{j=1}^n\varphi(s_{e_u}^{\pm 1}s_{f_j}s_{f_j}^*) \ (\text{by the defect free property})\\
&=\sum_{j=1}^n\varphi(s_{f_j}^*\sigma_{-i}(s_{e_u}^{\pm 1}s_{f_j})) \ (\text{by}\ \eqref{E:KMS}) \\
&=\sum_{j=1}^n{m^{\mp |u|}n^{-1}}\varphi(s_{f_j}^*s_{e_u}^{\pm 1}s_{f_j}) \ (\text{by}\ \eqref{E:sigma}) \\
&=\sum_{j=1}^n{m^{\mp |u|}n^{-1}}\varphi(s_{f_j}^*s_{f_j}s_{e_u}^{\pm 1}) \ (\text{as}\ \theta=\id) \\
&=m^{\mp |u|}\varphi(s_{e_u}^{\pm 1}).
\end{align*}

Therefore, $\varphi(X)=0$ for any generator $X\notin\fF$, and so $\varphi=\omega$.
\end{proof}

It immediately follows from \cite[Theorem 5.3.30]{BraRob} that

\begin{cor}
\label{C:idfactor}
$\pi(\O_\id)''$ is a factor.
\end{cor}

As we have observed, the propety $\theta=\id$ plays an important role in the the proof of Proposition \ref{P:uni_id}.
So the approach there can not be applied to the general case.
It is worth noticing that Corollary \ref{C:idfactor} can also follow from Proposition \ref{P:AFDid} below, which actually gives more.

\medskip

Now returning to the general case,
we have to assume that $\O_\theta^\sigma$ has a unique tracial state.
Then this unique tracial sate is nothing but $\tau\circ \Psi$,
where  $\Psi$ is the canonical faithful expectation from $\O_\theta^\sigma$ onto $\fF$, and $\tau$ is the
(unique) tracial state on $\fF$. Furthermore, simple calculations (on generators)
show that $\tau\circ \Psi=\omega|_{\O_\theta^\sigma}$.

\begin{prop}
\label{P:uni_Q}
Suppose that $\O_\theta^\sigma$ has a unique tracial state.
Then $\omega$ is the unique $\sigma$-KMS state over $\O_\theta$.
\end{prop}

\begin{proof}
From \cite{Yang}, it suffices to show that $\omega$ is the unique $\sigma$-KMS state over $\O_\theta$ when
$\frac{\ln m}{\ln n}\in\bQ$.

We now assume $\frac{\ln m}{\ln n}\in\bQ$. As before, let $a,b\in \bN$ with $\gcd(a,b)=1$ satisfy $m^a=n^b$.
Suppose that $\varphi$ is also a $\sigma$-KMS state over $\O_\theta$.
Notice that there is an obvious relation between $\sigma$ and the gauge action $\gamma$:
for all $t\in \bR$, $\sigma_t=\gamma_{(m^{-it}, n^{-it})}=\gamma_{(z^b,z^a)}$, where $z=\exp(-i\frac{t}{b})$.
Thus $\sigma$ gives an action of the torus $\bT$ on $\O_\theta$.
As in or \cite[Section 4.5]{BO} or \cite{DYperiod}, one can show that there is a faithful conditional expectation
$E$ from $\O_\theta$ onto $\O_\theta^\sigma$: $E(X)=\int_\bT\, \gamma_{(z^b,z^a)}(X)\, dz$ for all $X\in \O_\theta$. 
On one hand, from our hypothesis, it follows that $\phi\circ E=\tau\circ \Psi\circ E=\omega\circ E$,
where $\tau\circ\Psi$ is the unique tracial state of $\O_\theta^\sigma$ (see the remark preceding the statement of this theorem). 
On the other hand, since both $\phi$ and $\omega$ are $\sigma$-KMS states, they both are invariant with respect to $\sigma$.
So from the relation between $\sigma$ and $\gamma$ indicated as above, we infer that  
$\phi=\phi\circ E$ and $\omega\circ E=\omega$. Therefore, $\sigma=\omega$, proving the uniqueness of $\omega$.
\end{proof}

Some remarks are in order.
\begin{rem}
By Corollary \ref{C:indep}, our assumption is equivalent to requiring that $\O_\id^\sigma$ has a unique tracial state, namely,
$\tau$ has a unique extension from $\fF$ to $\O_\theta^\sigma$.
\end{rem}

\begin{rem}
\label{R:assump}

Our assumption seems to be quite reasonable although it looks strong.

\begin{itemize}
\item
First of all, if $\frac{\ln m}{\ln n}\notin\bQ$, the above assumption holds true automatically since $\O_\theta^\sigma=\fF$ in this case.
\item
Secondly, since $\O_\theta^\sigma\cong\fF\times_\rho \bZ$ (cf. Theorem \ref{T:crossprod}),
there are several known characterizations of
$\O_\theta^\sigma$ having a unique tracial state.

$\O_\theta^\sigma$ has a unique tracial state$\Longleftrightarrow$
$\rho$ is uniformly outer $\Longleftrightarrow$ $\rho$ is $\tau$-properly outer $\Longleftrightarrow$
$\omega|_{\O_\theta^\sigma}$ is a factor state $\Longleftrightarrow$ $\rho$ has the Rohlin property $\Longleftrightarrow$ $\O_\theta^\sigma$
is of real rank zero $\Longleftrightarrow$ the Connes spectrum $\Gamma(\tilde \rho)$ of $\tilde \rho$ is equal to $\bT$,
where $\tilde \rho$ is the automorphism of $\pi_\omega(\fF)$ extended from $\rho$. (Notice that $\Gamma(\rho)=\bT$.)
For more details about those equivalences, refer to, e.g.,  \cite{Bedos2, Krish2, Thom}.
\item
Thirdly, if $\pi_\omega(\fF)'\cap\pi_\omega(\O_\theta^\sigma)''=\bC 1$, then
$\pi_\omega(\O_\theta^\sigma)''\cong \pi_\omega(\fF)''\times_{\tilde \rho} \bZ$ is a factor, and so
$\O_\theta^\sigma$ has a unique tracial state from above (cf. \cite[Chapter XI]{Take}).
\end{itemize}
\end{rem}

We are now able to obtain the following promised characterizations.
\begin{thm}
\label{T:factor}
Suppose that $\O_\theta^\sigma$ has a unique tracial state. Then
the following statements are equivalent:
\begin{itemize}
\item[{(i)}]
$\Fth$ is aperiodic;
\item[{(ii)}]
$\pi(\O_\theta)''$ is a factor;
\item[{(iii)}]
$\omega$ is the unique $\sigma$-KMS state over $\O_\theta$.
\end{itemize}
\end{thm}

\begin{proof}
(i) $\Rightarrow$ (iii) is from Proposition \ref{P:uni_Q}; (iii) $\Rightarrow$ (ii) is from \cite[Theorem 5.3.30]{BraRob};
and (ii) $\Rightarrow$ (i) is from \cite{DYperiod}.
\end{proof}

\section{Types of $\pi(\O_\theta)''$}
\label{S:type}

In this section, using the results in Section \ref{S:KMS},
we will prove that,
under the assumption that $\O_\theta^\sigma$ has a unique tracial state,
there is a dichotomy for the type of the factor $\pi(\O_\theta)''$:
it is of either type III$_1$, or type III$_\lambda$ ($0<\lambda<1$).
This is dependent of whether the ratio $\frac{\ln m}{\ln n}$ being rational or not,
where $m$ (resp. $n$) is the number of $e$'s (resp. $f$'s) in the two
types of generators of the rank 2 graph $\Fth$.
When $\theta=\id$, it turns out that our assumption is redundant. 
In order to achieve this, we need to use a different
approach, which takes advantage of some known
properties of Cuntz algebras.

Before giving our main results in this section, we need to introduce more notation.
Let $\fM$ be a von Neumann algebra and $\phi$ a normal state of $\fM$.
By $\{\sigma_t^\phi:t\in \bR\}$, we denote the modular automorphism group of $\phi$.
Then the \textit{Connes invariant} $T(\fM)$ is defined as
$$T(\fM):=\{t\in\bR: \sigma_t^\phi\ \text{is inner}\},$$
and its orthogonal complement $\Gamma (\fM)$ is given by
$$\Gamma (\fM)=\{\lambda \in\bR: \sigma_t^\phi\ \text{is inner} \Rightarrow e^{i\lambda t}=1\}.$$
If $\fM$ is $\sigma$-finite, then the \textit{Connes spectrum $\S(\fM)$} is the intersection over all faithful
normal states of the spectra of their corresponding modular operators \cite{BraRob}.
Then one has the following relation:
$$\Gamma (\fM)=\ln(\S(\fM)\setminus \{0\}).$$
Connes classified type III factors as follows.
A type III factor $\fM$ is said to be of
\begin{alignat*}{3}
type\ III_0 \quad \text{if} & \quad \S(\fM)=\{0, 1\};\\
type\ III_\lambda \quad \text{if} &\quad  \S(\fM)=\{0, \lambda^n:n\in \bZ\} \quad (0<\lambda<1);\\
type\ III_1 \quad \text{if} &\quad  \S(\fM)=\{0\}\cup\bR^+.
\end{alignat*}
See, e.g., \cite{BraRob, KadRin, Take} for more information.

\subsection{The general case}
\label{SS:general}
We are now able to determine the types of $\pi(\O_\theta)''$.

\begin{thm}
\label{T:apefactor}
Let $\Fth$ be an aperiodic 2-graph.
\begin{itemize}
\item[(i)]
If $\frac{\ln m}{\ln n}\not\in\bQ$, then $\pi(\O_\theta)''$ is an AFD factor of type III$_1$.
\item [(ii)]
If $\frac{\ln m}{\ln n}\in\bQ$ and $\O_\theta^\sigma$ has a unique tracial state, then $\pi(\O_\theta)''$ is an AFD factor of type III$_{m^{-\frac{1}{b}}}$ (or III$_{n^{-\frac{1}{a}}}$), where
$a,b\in\bN$ satisfy $m^a=n^b$ and $\gcd(a,b)=1$.
\end{itemize}
\end{thm}

\begin{proof}
(i): It is proved in \cite{Yang}. (It is listed here just for completeness.)

(ii): By Theorem \ref{T:factor}, $\pi(\O_\theta)''$ is a factor. It is AFD since $\O_\theta$ is amenable (\cite{KumPask}).

Since $\O_\theta^\sigma$ has a unique tracial state, 
$\pi_\omega(\O_\theta^{\sigma})''$ is a factor. It follows
from \cite[Section 28.3]{Str} that the Connes spectrum coincides
with the spectrum of the modular operator of $\omega$. That is, $\S(\O_\theta'')=\rm{Sp}(\Delta_\omega)$.
Recall that from \cite{Yang} the modular operator $\Delta_\omega$ associated with $\omega$ is given by
$$
\Delta_\omega(s_us_v^*)=\bn^{d(v)-d(u)}s_us_v^*.
$$
Let $(a,b)\in\bN^2$ be such that $m^a=n^b$ and $\gcd(a,b)=1$.
Then it is  now easy to see that
\begin{align*}
\rm{Sp}(\Delta_\omega)
&=\overline{\{m^kn^\ell:k,\ell\in\bZ\}} \\
&=\overline{\{m^k(m^{\frac{a}{b}})^\ell:k,\ell\in\bZ\}}\\
&=\overline{\{(m^\frac{1}{b})^{a\ell+bk}:k,\ell\in\bZ\}}\\
&=\{0,(m^{-\frac{1}{b}})^N:N\in\bZ\}.
\end{align*}
So $\pi(\O_\theta)''$ is of type III$_{m^{-\frac{1}{b}}}$.
\end{proof}

\subsection{A Special Case: $\theta=\id$}
\label{SubS:Oid}

We shall see that, in this case,
the assumption -- $\O_\theta^\sigma$ has a unique tracial state --  is \textit{redundant}.
The approach here is different from the one used in Subsection \ref{SS:general}.
It makes use of some known properties of Cuntz algebras, and may be of independent interest.

\begin{prop}
\label{P:AFDid}
$\pi(\O_\id)''$ is an AFD factor. It is of type III$_1$ if $\frac{\ln m}{\ln n}\not\in\bQ$;
otherwise, it is of type III$_{m^{-\frac{1}{b}}}$, where $b$ is the same as that in Theorem \ref{T:apefactor}. 
\end{prop}

\medskip
\noindent
\begin{proof}

We first notice that $\O_\id\cong\O_m\otimes\O_n$,
where the *-isomorphism maps $s_{e_{u_1}f_{v_1}}s_{e_{u_2}f_{v_2}}^*$ to
$s_{e_{u_1}}s_{e_{u_2}}^* \otimes s_{f_{v_1}}s_{f_{v_2}}^*$.
See, e.g.,  \cite{DYperiod, Eva}.
(One should notice that there is no confusion for the tensor product used here since Cuntz algebras are amenable.)
Let $\omega_m:=\omega|_{\O_m}$ and $\omega_n:=\omega|_{\O_n}$.
Then it is not hard to see that $\omega_m$ and $\omega_n$ are
the distinguished states of $\O_m$ and $\O_n$, respectively.

In the sequel, we first claim that
$\omega=\omega_m\otimes \omega_n$.
Recall that $s_{e_i}$'s and $s_{e_i}^*$'s commute with $f_j$'s and $f_j^*$'s
in this case as $\theta=\id$.
Now let $X:=s_{e_{u_1}f_{v_1}}s_{e_{u_2}f_{v_2}}^*=s_{e_{u_1}}s_{e_{u_2}}^*s_{f_{v_1}}s_{f_{v_2}}^*$ be an arbitrarily standard
generator of $\O_\id$.
It follows from \cite[Lemmar 5.1]{Yang} that
\begin{align*}
\omega(X)
&=\omega(s_{e_{u_1}f_{v_1}}s_{e_{u_2}f_{v_2}}^*)  \\
&=\omega(s_{e_{u_1}}s_{f_{v_1}}s_{f_{v_2}}^*s_{e_{u_2}}^*)\\
&=\delta_{|u_1|,|u_2|}\,\delta_{|v_1|,|v_2|}\tau(s_{e_{u_1}}s_{f_{v_1}}s_{f_{v_2}}^*s_{e_{u_2}}^*)\\
&=\delta_{u_1,u_2}\,\delta_{v_1,v_2}\, m^{-|u_1|}\,n^{-|v_1|}\\
&=\omega(s_{e_{u_1}}s_{e_{u_2}}^*)\,\omega(s_{f_{v_1}}s_{f_{v_2}}^*)\\
&=\omega_m(s_{e_{u_1}}s_{e_{u_2}}^*)\,\omega_n(s_{f_{v_1}}s_{f_{v_2}}^*)\\
&=\omega_m\otimes \omega_n(X)
\end{align*}
Thus we conclude that $\omega=\omega_m\otimes \omega_n$.

Now, by virtue of \cite[Theorem 4.9]{Take},
the GNS representations induced by $\omega, \omega_m, \omega_n$ have the following relation:
$$\pi_\omega=\pi_{\omega_m}\otimes \pi_{\omega_n}.$$
Since Cuntz algebras are amenable, it follows from \cite[Proposition 4.13]{Take} that
\begin{align}
\label{E:ts}
\pi(\O_\id)''=\pi_{\omega_m}(\O_m)''\otimes \pi_{\omega_n}(\O_n)''=:\pi(\O_m)''\otimes\pi(\O_n)''.
\end{align}
It is well-known that $\pi(\O_m)''$ and $\pi(\O_n)''$
both are type III factors. So is $\pi(\O_\id)''$ by \cite[Corollary 11.2.17, Proposition 11.2.26 ]{KadRin}.
Therefore $\pi(\O_\id)''$ is an AFD  type III factor.

From \eqref{E:ts} and \cite[Corollary 13.1.17 ]{KadRin}, it yields
$$T(\pi(\O_\id)'')=T(\pi(\O_m)'')\cap T(\pi(\O_n)'').$$
But it is well-known that $T(\pi(\O_n)'')=\frac{2\pi}{\ln n}\bZ$ ($n\ge 2$) (cf., e.g., \cite{OlePed}).
This gives rise to
$$T(\pi(\O_\id)'')=\frac{2\pi}{\ln m}\bZ\bigcap \frac{2\pi}{\ln n}\bZ.$$
As $\pi(\O_\id)''$ is of type III, we have $0\in \S(\pi(\O_\id)'')$.
Therefore,
\begin{itemize}
\item[(i)]
if $\frac{\ln m}{\ln n}\not\in\bQ$, then
 $\Gamma(\pi(\O_\id)'')=\bR$, and so
 $\S(\pi(\O_\id)'')=[0,+\infty)$;
\item[(ii)]
if $m^a=n^b$ with $\gcd(a,b)=1$, then
 $\Gamma(\pi(\O_\id)'')=(\frac{1}{b}\ln m) \bZ$, and so $\S(\pi(\O_\id)'')=\{0,(m^{-\frac{1}{b}})^k:k\in \bZ\}$.
\end{itemize}

Of course, in Case (i), $\pi(\O_\id)''$ is of type III$_1$, while in Case (ii), $\pi(\O_\id)''$ is of type III$_{m^{-\frac{1}{b}}}$.
\end{proof}

\smallskip

It is worthwhile to point out that, for \textit{any} aperiodic 2-graph $\Fth$, one actually always has a *-isomorphism, say $\psi$, between
$\O_\theta$ and $\O_m\otimes \O_n $
by the Kirchberg-Phillips classification theorem (\cite[Corollary 5.3]{Eva}).
However, unlike the *-isomorphism between $\O_\id$ and $\O_m\otimes \O_n $, in general we do not have an explicit formula for 
$\psi$. In particular, we can not conclude $\omega=\omega_m\otimes \omega_n$ as in the proof of Proposition \ref{P:AFDid}. This is why the
approach used there can not be applied to the general case in Subsection \ref{SS:general}.

\begin{rem}
If the assumption that  $\O_\theta^\sigma$ has a unique tracial state were redundant, then one could obtain
the following nice results (cf. Theorem \ref{T:factor}, Theorem \ref{T:apefactor}, and \cite{DYperiod}):
\textit{
\begin{align*}
\Fth\ \mbox{is aperiodic}
&\Longleftrightarrow \O_\theta \ \text{is simple}\\
&\Longleftrightarrow \pi(\O_\theta)'' \ \text{is a factor}\\
&\Longleftrightarrow
\omega \ \text{is the unique} \ \sigma\text{-KMS state over}\  \O_\theta.
\end{align*}
Moreover, there is a dichotomy for the type of $\pi(\O_\theta)''$:
it is of type III$_{m^{-\frac{1}{b}}}$ (or III$_{n^{-\frac{1}{a}}}$) if $\frac{\ln m}{\ln n}\in\bQ$,
where $a,b\in\bN$ satisfy $m^a=n^b$ and $\gcd(a,b)=1$; it is of type III$_1$ otherwise.
}

\smallskip
So it would be interesting and nice to know if our assumption is redundant in general.

\end{rem}

\end{document}